\newif\ifpictures
\picturestrue

\documentclass[12pt]{amsart}
\usepackage{amssymb}
\usepackage[dvips]{graphicx}
\usepackage[latin1]{inputenc}
\usepackage{wasysym}
\usepackage[usenames]{color}
\ifpictures
\usepackage{psfrag}
\fi

\numberwithin{equation}{section}
\newtheorem{thm}{Theorem}
\newtheorem{prop}[thm]{Proposition}
\newtheorem{lemma}[thm]{Lemma}
\newtheorem{cor}[thm]{Corollary}

\theoremstyle{definition}

\newtheorem{example}[thm]{Example}
\newtheorem{remark}[thm]{Remark}
\newtheorem{openproblem1}[thm]{Open Problem}
\newtheorem{definition}[thm]{Definition}

\numberwithin{thm}{section}

\numberwithin{thm}{section}

\newcounter{FNC}[page]
\def\newfootnote#1{{\addtocounter{FNC}{2}$^\fnsymbol{FNC}$%
     \let\thefootnote\relax\footnotetext{$^\fnsymbol{FNC}$#1}}}

\newcommand{\C}{\mathbb{C}}
\newcommand{\N}{\mathbb{N}}

\newcommand{\R}{\mathbb{R}}

\newcommand{\abb}[5]{%
\setlength{\arraycolsep}{0.4ex}%
\begin{array}{rcccc}%
#1 &:\,& #2 & \,\,\longrightarrow\,\, & #3 \\[0.5ex]%
     & & #4 & \longmapsto & #5%
\end{array}%
}

\definecolor{RED}{rgb}{0.6,0,0}

\newcommand{\HH}{\mathcal{H}}

\DeclareMathOperator{\rank}{rank}

\title{On the degree and half degree principle for symmetric polynomials}

\author{Cordian Riener}

\address{C.~Riener, Institut f\"ur Mathematik,
Goethe Universit\"at , 60325 Frankfurt, Germany}
\email{riener@math.uni-frankfurt.de}

\begin{document}
\maketitle
\begin{abstract}
This note presents a new and elementary proof of a statement that was first proved by  Timofte\cite{tim}. It says that a symmetric real polynomial $F$ of degree $d$ in $n$ variables is positive on $\R^n$ ( on $\R^{n}_{+}$) if and only if it is so on the subset of points with at most $\max\{\lfloor d/2\rfloor,2\}$ distinct components.
The key idea of our new  proof lies in the representation of the  orbit space. The fact that for the case of the symmetric group $S_n$ it can be viewed as the set of normalized univariate real polynomials with only real roots allows us to conclude the theorems in a very elementary way.
\end{abstract}

\maketitle
\section{Introduction}
The question of certifying that a given polynomial in $n$ real variables is positive has been one of the main motivations for the development of modern real algebraic geometry in the beginning 20th century. Besides the general solutions  to this question by Hilbert, Artin and P\'olya only little interest has been devoted to the study of the related questions in the case of symmetric polynomials  (see \cite{CLR} and \cite{Pro} for example). However, in  \cite{tim}, Vlad Timofte was able to prove some fundamental properties of the positivity questions for symmetric polynomials with given degree:\\ 
For $n\in\N$  the group of all permutations of an $n$-element set is called the symmetric group $S_n$. This group acts on $\R^{n}$ in an obvious way: $\sigma(x_1,\ldots,x_n)=(x_{\sigma(1)},\ldots,x_{\sigma(n)})$ for $\sigma\in S_n$. Let $\R[X]:=\R[x_1,\ldots,x_n]$ denote the ring of polynomials in $n$ real variables. A polynomial $F\in\R[X]$  is called \emph{symmetric}, if for all $\sigma\in S_n$ we have $F(x)=F(\sigma(x))$. We will write $\R[X]^{S_n}$ for the ring of symmetric polynomials. The essence of the main theorems we present in this paper is that in order to check if a symmetric polynomial (in-) equality is valid one only needs to check if it is valid on test sets of dimension (half) degree of the polynomial. More precisely: 
Let $x\in\R^n$ and let $n(x)=\#\{x_1,\ldots,x_n\}$ denote the number of distinct components of $x$ and $n^{*}(x)=\{x_1,\ldots,x_n\,|x_j\neq 0\}$ denote the number of distinct non zero elements. Then for a given $d\in\N$ we will take a look at sets of the form $A_d:=\{x\in\R^{n}: n(x)\leq d\}$ i.e. the points in $\R^{n}$ with at most $d$ distinct components and sets $A_d^{+}:=\{x\in\R^{n}_{+}: n^{*}(x) \leq d\}$ i.e points with at most $d$ distinct non zero elements. With this setting Timofte discovered the following remarkable theorems.
\begin{thm}\label{degree}
\emph{[Degree principle~]}
Let $F\in\R[X]^{S_n}$ be of degree $d$. Then there is $x\in\R^{n}$ with $f(x)=0$ if and only if there is $y\in A_d$ with $f(y)=0$
\end{thm}
\begin{remark}
Instead of one polynomial, one can also look at a system of symmetric polynomials $F_1,\ldots,F_k$ of degree at most $d$. The proof of theorem \ref{degree} shows that in this case the corresponding real variety $V_{\R}(F_1,\ldots,F_k)$ will be empty if and only if $V_{\R}(F_1,\ldots,F_k)\cap A_d$ is empty.
\end{remark}
The second statement involves inequalities and is even less expected:
\begin{thm}\label{halfdegree}
\emph{[Half degree principle]}
Let $F\in\R[X]^{S_n}$ be of degree $d$ and let $k:=\max\{2,\lfloor\frac{d}{2}\rfloor\}$. Then the inequality $F(x)\geq 0$ holds on $\R^{n}$ (resp. on the positive orthant $\R^{n}_{+}$) if and only if it holds on $A_{k}$ (resp. on $A_{k}^{+}$)
\end{thm}

The original proofs of these results relied mostly on the existence of a solution to a differential equation and did not fully capture the geometric picture that plays in fact a key role as we intend to show in this article. Hence, instead of the purely analytic way, we will provide proofs that exploit some underlying geometric properties.  

This article will be structured as follows: In the next section we will give some  background from the theory of symmetric polynomials and the geometry of the so called Orbit variety. As in the case of the symmetric group $S_n$,  the orbit space of $\R^{n}$ can be seen as the space of univariate polynomials of degree $n$ with only real roots and hence some very elementary properties of such polynomials will be presented in section 3. After section 3 we will be able to give a short and elementary proof of the main theorems using the viewpoint presented in section 2. To make this article as self contained as possible we will provide short proofs to all statements needed.

\section{Symmetric polynomials and the orbit Space of $S_n$}
Among the polynomials that are invariant to the action of the symmetric group the following two families are of special interest:
\begin{definition}
For $n\in \N$, we consider the following two families of symmetric polynomials.
\begin{enumerate}
 \item For $k\leq n$ let $p_k:=\sum_{i=i}^{k}x_i^k$ denote the $k$-th power sum polynomial
 \item For $k\leq n$ let $e_k:=\sum_{1\leq i_1<i_2<\ldots <i_k\leq n} x_{i_1}x_{i_2}\cdots x_{i_k}$ denote the $k$-th elementary symmetric polynomial
\end{enumerate}
\end{definition}
These two families of symmetric polynomials are linked by the so called Newton identities ( see e.g.\cite{M}):
\begin{equation}
 k(-1)^{k} e_k(x)+\sum_{i=1}^{k}(-1)^{i+k}p_{i}(x)e_{k-i}(x)=0
\end{equation}
One of the things that mark the importance of these two families is that both of them are generators of the algebra $\C[x]^{S_n}$.
\begin{thm}
The ring of symmetric polynomials $\C[X]^{S_n}$ is a polynomial ring in the n elementary symmetric polynomials $e_1,\ldots,e_n$.
\end{thm}
Although this statement is rather classical we provide a short proof from which we  then deduce  more information about the expression of a symmetric polynomial of given degree in terms of the elementary symmetric polynomials. The proof follows the exposition given in \cite{Stu}.
\begin{proof}
Let $F$ be a symmetric polynomial and we compare the monomial involved in $F$ using lexicographic order on the degrees i.e. $x_1^{\alpha_1}\cdots x_n^{\alpha_n} \geq_{Lex} x_1^{\beta_1}\cdots x_n^{\beta_n}$ if $\sum\alpha_i > \sum\beta_i$ or if the first non zero element of the sequence $(\alpha_i-\beta_i)$ is positive.

Let $a\cdot x_1^{\gamma_1}\cdots x_n^{\gamma_n}$ be the biggest monomial with respect to the Lex-order. As $F$ is supposed to be symmetric it follows that $\gamma_1\geq\gamma_2\geq\cdots\geq\gamma_n$. Now we consider the polynomial
$H:=a\cdot e_1^{\gamma_2-\gamma_1}\cdot e_2^{\gamma_3-\gamma_2}\cdots e_n^{\gamma_n}$. The greatest monomial of $H$ is equal to $a\cdot x_1^{\alpha_1}\cdots x_n^{\alpha_n}$ hence if we consider $f\tilde{F}=F-H$ this term will get lost. Now we can use the same arguments with $\tilde{F}$. As the leading monomial of each step will be canceled, this procedure will terminate and give us a description of $F$ as a polynomial in the elementary symmetric polynomials $e_1,\ldots,e_n$. It remains to show that this representation is unique, i.e. that $e_1,\ldots,e_n$ are really algebraically independent. Suppose, that there is $0\neq G\in\R[z_1,\ldots,z_n]$ such that $g(e_1(x),\ldots,e_n(x))$ is identically zero. Now consider any monomial $z_1^{a_1}\cdots z_n^{a_n}$ of $G$. Then the initial monomial of $e_1^{a_1}\cdots e_n^{a_n}$ will be $x_1^{a_1+a_2+\ldots+a_n}x_2^{a_1+a_2+\ldots+a_n}\cdots x_n^{a_1+a_2+\ldots+a_n}$. But as the linear map
$$(a_1,\ldots,a_n)\mapsto (a_1+\ldots + a_n, a_2+\ldots+ a_n,\ldots, a_n)$$
is injective, all other monomials of $G$ will have different initial monomials. The lexicographically largest monomial is not cancelled by any other monomial, and therefore $G(e_1,\ldots,e_n)\neq 0$. 
\end{proof}
\begin{remark}
We can replace $\C$ in the above theorem with any other field. 
\end{remark}
Let $F$ now be a  given a real symmetric polynomial of degree $d\leq n$ and let $G\in\R[z_1,\ldots,z_n]$ be the corresponding polynomial in the elementary symmetric polynomials. Under these circumstances the above proof will also tell us something about the possible monomials that are involved in $G$, namely we can easily deduce the following three statements:
\begin{enumerate}
 \item There will be no monomial that contains a variable $z_j, j>n$.
 \item There will be no monomial that contains two variables $z_j,z_i$ with $i,j\geq \lfloor\frac{d}{2}\rfloor$.
 \item The variables $z_j$ with $i\geq \lfloor\frac{d}{2}\rfloor$ occur at most linearly in every monomial.
\end{enumerate}
Summing up the above statements $G$ can be written uniquely as
\begin{equation}
 G(z_1,\ldots,z_n)=G_1(z_1,\ldots,z_{\lfloor\frac{d}{2}\rfloor})+\sum_{i=\lfloor\frac{d}{2}\rfloor}^d G_i(z_1,\ldots z_{d-i})z_i\label{eq:darst}
\end{equation}\label{eq:G}

Whereas the last two properties of $G$ will play a role in the derivation of the half degree principle, the first is in fact the heart of the degree principle. 
A very nice way to see what is going on if one passes from $F$ to the polynomial $G$ was first pointed out by Procesi in his paper\cite{Pro}:\\
Every $x\in\C^{n}$ can be viewed as  the $n$ roots of the univariate polynomial $$f(t)=\prod_{i=1}^{n}(t-x_i).$$ The classical Vieta formula implies, that $f(t)$ can also be written as $$f(t)=x^{n}-e_1(x)x^{n-1}+\ldots\pm e_n(x).$$ Using geometric language the identification of the $n$ roots with the $n$ coefficients can be thought of as  giving rise to an surjective  map 
$$\abb{\pi}{\C^n}{\C^n}{x:=(x_1,\ldots,x_n)}{\pi(x):=(e_1(x),\ldots,e_n(x))}.$$

Obviously $\pi$ is constant on $S_n$ orbits and hence the ring $\C[X]^{S_n}$ is exactly the coordinate ring of the image of $\pi$ called the \emph{orbit space}.

 It is worth mentioning that $\pi$ has very nice continuity properties: Obviously the coefficients of a univariate polynomial $f$ depend continuously on the roots, but also the converse is true:
\begin{thm}\label{thm:cont}
Let $f=\prod_{i=1}^{k}(t-x_i)^{m_i}=\sum_{j=0}^{n}a_jx^{j}$ be a univariate polynomial and define $0<\epsilon<|\min_{i\neq j} x_i-x_j|/2$. Then there is a $\delta>0$ such that every polynomial $g=\sum_{j=0}^{n}b_jx^j$ with coefficients satisfying $|a_j-b_j|< \delta$ has exactly $m_i$ zeros in the disk around $x_i$ with radius $\epsilon$.
\end{thm}
\begin{proof}
See for example \cite{RS} (Thm. $1.3.1$) .
\end{proof}

As we want to know about real zeros of the polynomial $F$ we will have to restrict  $\pi$ to $\R^{n}$. 
In this case the restriction maps into $\R^{n}$ but it fails to be surjective: Already the easy example $x^2+1$ shows that we can find $n$ real coefficients that define a polynomial with strictly less than $n$ real zeros. Polynomials with real coefficients that only have real roots are sometimes called \emph{hyperbolic}. The right tool to characterize the univariate hyperbolic polynomials is the so called Sylvester-Matrix:\\
Let $K$ be any field and take $f(t)=t^n+b_1t^{n-1}+\ldots+a_n\in K[x]$ a univariate normalized polynomial. Its $n$ zeros $\alpha_1,\ldots,\alpha_n$ exists in the algebraic closure of $K$. For $r=0,1,\ldots $ let $p_r(f):=\alpha_{1}^r+\ldots+\alpha_{n}^{r}$ be the $r$-th power sum evaluated at the zeros of $f$.
Although it seems that this definition involves the a priori not known  algebraic closure of $K$ and the roots of $f$, which are also not known a priori, these numbers are well defined. 
We have $p_{r}(f)\in K$ and using Vieta and the Newton relations, we can express the power sums as polynomials in the coefficients of $f$. 
\begin{definition}
 The Sylvester Matrix $S(f)$ of a normalized univariate polynomial of degree $n$ is given by
$$S(f):=(p_{j+k-2}(f))_{j,k=1}^{n}$$
\end{definition}
Without too much abuse of notation we will use $S(z)$  for every  $z\in R^{n}$ to denote the Sylvester Matrix of corresponding polynomial whose coefficients are $z$.

Now the key observation we will need is Sylvester's version of Sturms theorem.
\begin{thm}\label{syl}
Let $R$ be a real closed field and $f\in R[t]$ a normalized polynomial of degree $n\geq 1$. 
\begin{enumerate}
 \item The rank of $S(f)$ is equal to the number of distinct zeros of $f$ in the algebraic closure $R(\sqrt{-1})$.
 \item The signature of $S(f)$ is exactly the number of real roots of $f(x)$.
\end{enumerate}
\end{thm} 
Using the above theorem we see that $f\in\R[t]$ is hyperbolic if and only if $S(f)$ is positive definite (denoted by $S(f)\succeq 0$).
With machinery of hyperbolic polynomials we are now able to understand the situation and we can sum it up in the following theorem which was noted by Procesi \cite{Pro}:
\begin{thm}
Let $F\in\R[X]^{S_n}$ and $G\in\R[z_1,\ldots z_n]$ be the corresponding polynomial according to equation $(2.2) $- then for any  $b\in\R$ the following are equivalent:
\begin{enumerate}
\item There is  $x\in\R^{n}$ such that $ F(x)=b$ 
\item There is  $z\in \R^{n}$ such that the polynomial $t^n-z_1t^{n-1}+\ldots\pm z_n$ is hyperbolic and $G(z)=b$.
\item There is $z\in\R^{n}$ such that $S(z) \succeq 0$ and $G(z)=b$
\end{enumerate}
\end{thm}
Now the strategy in order to prove Theorem \ref{degree} and \ref{halfdegree} is to take the view point of the orbit space. Instead of $F$ on $\R^{n}$, we will have to examine $G$ over the set $$\mathcal{H}:=\{z\in\R^{n}:\,t^n-z_1t^{n-1}+\ldots\pm z_n \text{ is hyperbolic}\}$$and the sets  $$\mathcal{H}^{k}:=\{z\in\mathcal{H}:\,t^n-z_1t^{n-1}+\ldots\pm z_n \,\text{ has at most } k \text{ distinct zeros}\}.$$
\begin{remark}
 We observe  from theorem \ref{syl}  that the sets $\mathcal{H}$ and $\mathcal{H}^{k}$ are closed semi algebraic sets. 
\end{remark}
We will have to show, that 
\begin{equation}\label{eq:degree}
G(\mathcal{H})=G(\mathcal{H}^{d}),
\end{equation}
in order to prove the degree principle, the half degree principle follows from
\begin{equation}\label{eq:halfdegree} 
\min_{z\in\mathcal{H}}G(z)=\min_{z\in {H}^{\lfloor d/2\rfloor}}G(z).
\end{equation}

In order to do this examination of $G$ in an easy way, we will need some very elementary facts about polynomials with only real roots. We will show these facts about hyperbolic polynomials in the next section.

\section{Hyperbolic polynomials}

The main problem that we will have to  deal with in order to prove the main theorems is the question which changes of the coefficients of a hyperbolic polynomial will result in polynomials that are still hyperbolic. This question is in fact very old and has already been studied by P\'olya,  Schur (see for example \cite{PSch} and\cite{PS}) However we will only need very simple results. All these results are in fact based on the classical Rolle's theorem:
\begin{thm}
 Let $f\in\R[t]$ and $a,b\in\R$ with $a<b$ and $f(a)=f(b)=0$. Then the derivative polynomial $f^{'}(t)$ has a root in $(a,b)$.
\end{thm}
From this classical result we can deduce some very helpful corollaries:
\begin{cor}\label{cor:wichtig}
Let $f=t^n+a_1t^{n-1}+\ldots +a_n$ be hyperbolic. Then the following hold: 
\begin{enumerate}
\item Let $a,b\in\R$ with $a\leq b$. If $f$ has $d$  roots (counted withmultiplicitiess) in $[a,b]$ then $f^{'}$ has at least $d-1$ roots in $[a,b]$.
\item All derivatives of $f$ are also hyperbolic.
\item There is no local maximum $\xi_1$ of $f$ such that $f(\xi_{1})<0$ and no local minimum $\xi_2$ with $f(\xi_2)>0$.
\item If $f$ as only distinct roots, then there is a $\delta >0$ such that for all $0<\varepsilon<\delta $ the polynomial $f\pm\varepsilon$ is also hyperbolic with $n$ distinct roots.
\item The multiple zeros of its derivative are multiple zeros of $f$.
\item If $a_i=a_{i+1}=0$ then $a_{j}=0$ for all $j\geq i$  
\end{enumerate}
\end{cor}
\begin{proof}
\begin{enumerate}
\item If $a=b$ then $f$ has a multiple root of order $d$ at $t=a$. Hence its derivative has a multiple root of order $d-1$ at $t=a$. If $a<b$ let $t_1,\ldots t_k$ be the different roots of $f$ and $d_1,\ldots,d_k$ the corresponding multiplicities. Now at each $t_i$ the derivative $f^{'}$ has a root of order $d_i-1$. Further from Rolle's theorem we see that $f^{'}$ has a root in each open interval $(t_i,t_{i+1})$. Hence in total $f^{'}$ has at least $d_1-1+d_2-1+\ldots+d_k-1+(k-1)=d-1$ zeros. 
\item $f$ has $n$ zeros on the real line and using the previous we see that $f^{'}$ has its $n-1$ roots there. Now the same argument holds for the other derivatives.
\item The local extrema of $f$ are exactly the zeros of its derivative. But then the statement is obvious from the last two. 
\item Let $\xi_1,\ldots, \xi_{n-1}$ be the zeros of $f'$. Then define $\delta:=\min\{f(\xi_1),\ldots,f(\xi_{n-1})\}$. Then for $0<\varepsilon\delta$ every polynomial $f\pm\varepsilon$ will have the same derivative polynomial and therefore also the same local extrema. By construction of $\delta$ we have that $f\pm\varepsilon$ will be negative on all local minima but positive on all local maxima. Therefore $f\pm \varepsilon$ has $n$ real roots.
\item Otherwise the number of roots does not match.
\item If $a_i=a_{i+1}=0$ there is a derivative of $f$ with a multiple root at $t=0$. But then $t=0$ is also a multiple root of $f$ of order $n-i+1$ hence $a_j=0$ for all $j\geq i$. 
\end{enumerate}
\end{proof}

As already mentioned we want to know, which small perturbations of coefficients of a hyperbolic polynomial will result in a hyperbolic one.
The above corollary already gave us that we can perturb the constant coefficient if all zeros are distinct.
The following  easy constructions  will allow us to determine which coefficients can be perturbated if a hyperbolic polynomial $f$ has $k$ distinct roots.
\begin{prop}\label{prop:wichtig} 
Let $f\in\R[t]$ be a hyperbolic polynomial of degree $n$ with $k<n$ different zeros.
 Then for each $1\leq s\leq k$ there is a polynomial $g_s$ of degree $n-s$ and a $\delta_{s}> 0$ such that for all $0<\epsilon <\delta_{s}$ the polynomials $f\pm\epsilon g$  are also hyperbolic and have strictly more distinct zeros.
\end{prop}
As this proposition it in fact the heart of our reasoning we will provide an elementary constructive  proof:

\begin{proof}
Let $x_1,\ldots,x_k$ be the distinct zeros of $f$ and assume that $x_j$ is a multiple root. 

 \item We can factor $$f=\underbrace{\prod_{i=1}^{s}(t-x_i)}_{:=p(t)}\cdot g_1(t),$$
where the set of zeros of $g_1$ contains only elements from $\{x_1,\ldots x_k\}$ and $g_1$ is of degree $n-s$. Now we can apply \ref{cor:wichtig} (4) to see that $p(t)\pm \varepsilon_k$ is hyperbolic. Furthermore we see that $p(t)\pm \varepsilon_k$ has none of its roots in the set $\{x_1,\ldots,x_k\}$. Hence $(p(t)\pm\varepsilon_k)\cdot g_1=f(t)+\varepsilon_k g_1$ is hyperbolic and has more than $k$ different roots.

\end{proof}

As we also want to prove the half degree principle for $\R_{+}^{n}$ the following easy observation will also be useful:
\begin{prop}
The map $\pi$ maps $\R^n_{+}$ onto $\mathcal{H}_{+}:=\R^{n}_{+}\cap \mathcal{H}$.
\end{prop}
\begin{proof}
It is easy to see that $\pi(\R^{n}_+)\subseteq\R^n_+$: If $x\in\R^n_+$ all $e_i(x)$ are also positive.

To see the other inclusion: Lets assume that $x\in\R^n$ has at least one negative component. If there is an odd number of such components then of course $e_n(x)=x_1\cdots x_n$ is negative and we have a contradiction. If the number is even take the derivative of the associated polynomial. Its $n-1$ roots $\tilde{x}_1,\ldots,\tilde{x}_{n-1}$ lay interlacing between the $x_i$. Hence there is at least one negative component. As thecoefficientss of a polynomial and its derivative  just differ by positive factors we have that $e_i(\tilde{x})<0$ if and only if $e_i(x)<0$. So if the number ofnegativee components of $\tilde{x}$ is odd, we are done. If not we derivate again until we get a contradiction. 
\end{proof}
By definition of the set $\mathcal{H}_{+}$ it could be possible that there are all sorts of polynomials with zero coefficients. But for our transfer of the half degree principle to $\mathcal{H}_+$ we will need the following easy proposition:
\begin{prop}\label{prop:zeros}
Let $f:=t^n+a_1t^{n-1}+\ldots+a_n$ be a hyperbolic polynomial with only positive roots. If $a_{n-i}=0$ for one $i$ then $a_j=0$ for all $j\leq i$.
\end{prop}
\begin{proof}
First recall that if $f$ has only positive roots, all its derivatives share this property. If $a_{n-i}=0$ we know that the $i$th derivative of $f$ has a root at $t=0$. But as the $i-1$-th derivative of $f$ has also only positive roots, also it needs to have a root at $t=0$. Now the statement follows since this implies that $f$ has a multiple root of order $i$ at $t=0$. 
\end{proof}
To study the polynomials on the boundary of $\mathcal{H}_{+}$ the following consequence of proposition \ref{prop:wichtig} will be helpful:
\begin{prop}\label{prop:zeros2}
Let $f\in\R[t]$ be a hyperbolic polynomial of degree $n$ with $k<n$ different zeros with an $k>m$-fold root at $t=0$.
Then for each $1\leq s\leq k$ there is a polynomial $g_s$ of degree $n-s$ with $m$-fold root at $t=0$ and a $\delta_{s}> 0$ such that for all $0<\epsilon <\delta_{s}$ the polynomials $f\pm\epsilon g$  are also hyperbolic and have strictly more different zeros.
\end{prop}
\begin{proof}
Just consider the hyperbolic polynomial $\tilde{f}:=\frac{f}{x^{m}}$ of degree $n-m$ with $k-m$ distinct zeros. Applying \ref{prop:wichtig} to $\tilde{f}$ we get $\tilde{g}_s$ of degree $n-m-s$ but then obviously $g_s:=\tilde{g}_s x^m$ meets the announced statements.
\end{proof}

\section{Elementary proofs for the degree and half degree principle}
This last section   uses the  statements about univariate polynomials given in the previous section to prove the main statements.  The proofs   will be based on a very elementary optimization problem. In order to introduce this problem we will first give some notation:

Recall that to each   $S_n$ orbit  of any  $x\in\R^{n}$ we associate the polynomial $$f(t)=\prod(t-x_i)=\sum_{i=0}^{n}a_it^{n-i}.$$
Then the set $$\mathcal{H}_{s}(a_1,\ldots,a_s):=\{z\in\R^n z_1=a_1,\ldots ,z_{s}=a_{s}, S(z)\succeq 0 \}$$ can be identified with the set of   all normalized  hyperbolic polynomials of degree $n$ that agree with $f$ on the leading $s+1$ coefficients.

Now for both the proof of the degree and the proof of the half degree principle will take a look at optimization problems of the following form:
\begin{eqnarray}\label{eq:op}
&\min c^{t}z\\
z&\in \mathcal{H}_{s}(a_1,\ldots,a_s),
\end{eqnarray}
where $c\in\R^{n}$ defines any linear function and $a_1,\ldots,a_s$ are fixed. To make the later argumentation easier, we set the minimum of any function over the empty set to be infinity.

A priori it may not be obvious that such problems have an optimal solution. But, this is a consequence of the following proposition:

\begin{prop}
For any $s\geq 2$ every set $\mathcal{H}_s(a_1,\ldots,a_n)\neq\emptyset$ is compact.
\end{prop}
\begin{proof}
A set  defined by $p_2(x)=a_2$ is a ball and compact.  The map $\pi$ is continuous  and therefore also the image of such sets, which are given by $z_1^2-2z_2=a_2$ are compact.  For $s\geq 2$ every $\mathcal{H}_s(a_1,\ldots,s_s)$ is contained in such a set and closed and therefore compact.
\end{proof}
Recall from theorem \ref{syl} that the points $z\in\R^{n}$ that define hyperbolic polynomials with exactly $k$ distinct roots are precisely those with $\rank S(z)=k$. We will use $\mathcal{H}_s^k(a_1,\ldots,a_s)$ to refer to those points in $\mathcal{H}_s(a_1,\ldots,a_s)$ where $\rank S(z)\leq k$, i.e. to those normalized hyperbolic polynomials which have at most $k$ distinct zeros an prescribed coefficients $a_1,\ldots,a_s$.

The  crucial observations, which will be the core of the theorems we want to prove lies, in the geometry of the optimal points of the above optimization problems. This is noted in the following lemma: 
\begin{lemma}\label{two}
Let $c\in\R^{n}$, $s\in\{1,\ldots, n\}$. Then $$\min_{z\in H_{s}(a_1,\ldots,a_s)} c^tz=\min_{z\in H_{s}^{s}(a_1,\ldots,a_s)}c^{t}z$$ 
\end{lemma}
\begin{proof}
If $c_i=0$ for all $i>s$ the linear function $c^tz$ is constant over $\HH_{s}(a_1,\ldots,a_s)$ and the statement follows in this case.
So let us assume that there is at least one $i>s$ with $c_i\neq 0$ and let $\tilde{z}_1\in\R^{n}$ with$$c^{t}\tilde{z}_1=\min_{z\in \HH_{s}(a_1,\ldots,a_s)} c^tz.$$ If $\rank S(\tilde{z_1})\leq s$ we are done.\\ 
So we assume by contrary that $\rank S(\tilde{z_1})=k>s$. Using proposition \ref{prop:wichtig} we see that there is $0\neq\tilde{y}\in{0}^s\times\R^{n-s}$ such that $\tilde{z}_1\pm\varepsilon \tilde{y}\in \HH_{s}(a_1,\ldots,a_s)$ for small enough positive $\varepsilon$.
Now if $c^t\tilde{y}\neq 0$ one of $\tilde{z}_1+\varepsilon \tilde{y}$ or $\tilde{z}_1-\varepsilon \tilde{y}$ will give a smaller value to the objective function which clearly contradicts the optimality of $\tilde{z}_1$. In the other case if $c^t\tilde{y}=0$ we observe that for $\tilde{z}_2:=\tilde{z}+\varepsilon \tilde{y}$ we have from \ref{prop:wichtig} $\rank S(\tilde{z}_2)>k$ and we can redo the above argumentation with $\tilde{z}_2$. Doing this we will either end up with $\tilde{z}$ which gives a smaller value or after finally  many iterations of this procedure at a point $\breve{z}$ with $\rank S(\breve{z})=n$. But then $\breve{z}$ lies in the relative interior of $\HH_{s}(a_1,\ldots,a_s)$ and therefore either the value $c^t\breve{z}$ is not the optimal value or all $c_i$ with $i>s$ must be equal to zero and we get a contradiction.

\end{proof}

From the above lemma we can conclude the following important corollary:
\begin{cor}\label{cor:degree}
Every set $\mathcal{H}_s(a_1,\ldots, a_s)\neq\emptyset$ with $s\geq 2$ contains a point $\tilde{z}$ with $\rank S(\tilde{z})\leq s$.
\end{cor}
\begin{proof}
Take  $c\in\R^{n}$ with $c_i=0$ for all $i\neq s$ Then the function $c^tz$ will not be constant over $\mathcal{H}_{s}(a_1,\ldots,a_s)$.  But as $\mathcal{H}_{s}(a_1,\ldots,a_s)$ is compact we know the minimal value is attained and we can conclude with lemma \ref{two}.
\end{proof}
To transfer the half degree principle to $\R^{n}_{+}$ we will also need to know what happens to the minima when we intersect a set $\mathcal{H}_s(a_1,\ldots, a_s)$ with $\R_{+}^{n}$. 
We denote this intersection with $\mathcal{H}^{+}_s(a_1,\ldots, a_s)$ and define  $$\mathcal{H}_{s}^{(s,+)}(a_1,\ldots,a_s):=\{z\in\mathcal{H}^{+}_s(a_1,\ldots, a_s): \rank S(z)\leq s\}
\cup\HH(a_1,\ldots,a_s,0,0,\ldots,0).$$

With these appropriate notations we  have a same type of argument as in lemma \ref{two}:
\begin{lemma}\label{lemma:three}
Let $c\in\R^{n}$, $s\in\{1,\ldots, n\}$. Then $$\min_{z\in H_{s}^{+}(a_1,\ldots,a_s)} c^tz=\min_{z\in H_{s}^{(s,+)}(a_1,\ldots,a_s)}c^{t}z.$$ 
\end{lemma}
\begin{proof}
The argument works out almost the same way as in lemma \ref{two}: Indeed if $z\in \HH_{s}^{*}(a_1,\ldots,a_s)$ has strictly positive components small perturbations of these will not change the positivity and the same arguments can be used. So just the cases of $z\in\HH^{+}(a_1,\ldots,a_s)$ with zero components need special consideration. So assume we have a $\tilde{z}\in\mathcal{H}(a_1,\ldots,a_s)$ with zero components   such that $c^t\tilde{z}\min_{z\in H^{+}_{s}(a_1,\ldots,a_s)} c^tz$. But  with proposition \ref{prop:zeros} we see that there is $i\in\{1,\ldots,n\}$ such that $\tilde{z}_j=0$ for all $j\geq i$. 
If $i\leq s+1$ we have already that that $\tilde{z}\in\mathcal{H}_{s}^{(s,+)}(a_1,\ldots,a_s)$ But if $s+1<i$ we can see from  \ref{prop:zeros2}  that there is $0\neq\tilde{y}\in{0}^s\times\R^{i-s}\{0\}^{n-i}$
such that $\tilde{z}_1\pm\varepsilon \tilde{y}\in H_{s}(a_1,\ldots,a_s)\cap \R^{N}_{+}$ for small positive $\varepsilon$ and  argue as in the previous lemma. 
\end{proof}

Now to conclude we can easily show the degree and the half degree principle in the following version:
\begin{thm}
Let $F\in\R[X]^{S_n}$ of degree $d\geq 2$, $G\in\R[z_1,\ldots z_n]$ be the corresponding polynomial according to equation \ref{eq:G} and set $k:=\max\{2,\lfloor d/2\rfloor\}$.
\begin{enumerate}
\item  We have $\exists z\in\mathcal{H}$ with $G(z)=0$ if and only if $\exists z\in\mathcal{H}$ such that $G(z)=0$.
\item We have $G\geq 0$ for all $z\in\mathcal{H}$ if and only if $G\geq 0$ for all $z\in\mathcal{H}^{k}$.
\item We have $G\geq 0$ for all  $z\in\mathcal{H}^{+}$ if and only if $G\geq 0 $ for all $z\in\mathcal{H}^{k,+}$.
\end{enumerate}
\end{thm}
\begin{proof}
\begin{enumerate}
\item We know from \ref{eq:G} that $G$ is constant on any set $\mathcal{H}_d(a_1,\ldots,a_d)$. As we have $$\bigcup_{(a_1,\ldots ,a_d)\in\R^{d}}\mathcal{H}_s(a_1,\ldots ,a_d)=\mathcal{H},$$ the statement and hence the degree principle (\ref{eq:degree})  follows now directly from corollary \ref{cor:degree}.
\item We will have to see that 
$$\min_{z\in\mathcal{H}\subset \R^{n}}G(z)=\min_{z\in \mathcal{H}^{k}}G(z).$$
Again we decompose the space in the form:
$$\bigcup_{(a_1,\ldots ,a_{k})\in\R^{k}}\mathcal{H}_s(a_1,\ldots ,a_{k})=\mathcal{H}$$
Therefore $$\min_{z\in\mathcal{H}}G(z)=\min_{a_1,\ldots,a_{k}}\min_{z\in \mathcal{H}(a_1,\ldots,a_{k})}G(z).$$
But for fixed $z_1=a_1,\ldots, z_{k}=a_{k}$ the function $G(z)$ is just linear and now we can apply lemma \ref{two} and see that: 
$$\min_{z\in \mathcal{H}(a_1,\ldots,a_{k})}G(z)=\min_{z\in \mathcal{H}^{k}(a_1,\ldots,a_{k})}G(z).$$
and we get \ref{eq:halfdegree}. Hence $G(z)$ is positive on $\mathcal{H}$ if and only if $G(z)$ is positive on $\mathcal{H}^{k}$ and the half degree principle is proved.
\item Again the function $G$ is linear over the sets $\mathcal{H}^{+}(a_1,\ldots,a_{k})$ and we can argue as above by using lemma \ref{lemma:three}. \end{enumerate}
\end{proof}
{\bf Acknowledgement} The author is very grateful to Markus Schweighofer and Thorsten Theobald for many helpful discussions and numerous comments.  

 \end{document}